\documentclass{article}

\usepackage{mathrsfs}
\usepackage{amsfonts}
\usepackage{amsmath}
\usepackage{amsthm}
\usepackage[all]{xy}
\usepackage{sectsty}
\usepackage{titlesec}
\usepackage{amssymb}
\usepackage{graphicx}
\newtheorem{theorem}{Theorem}[section]

\newtheorem{lemma}[theorem]{Lemma}
\newtheorem{proposition}[theorem]{Proposition}

\theoremstyle{plain}

\newcommand{\eb}{\begin{enumerate}}
\newcommand{\ee}{\end{enumerate}}
\newcommand{\ebx}{\begin{equation*}}
\newcommand{\eex}{\end{equation*}}

\newcommand{\beq}{\begin{eqnarray*}}
\newcommand{\eeq}{\end{eqnarray*}}

\makeatletter

\newcommand{\Rmnum}[1]{\expandafter\@slowromancap\romannumeral #1@}
\makeatother

\begin{document}

\centerline{\Large{A weaker rigidity theorem for pairs of hyperquadrics and its application}}

\bigskip

\centerline{Abstract}

\footnotesize{In this short article, we establish a rigidity theorem for pairs of hyperquadrics in a weaker sense, i.e., we impose a condition that minimal rational curves are preserved, which is stronger than inheriting a sub-VMRT structure, a notion raised by Mok \& Zhang (2014) (cf. \cite{[MoZ]}). This problem has its source in a theorem of Tsai (1993) (\cite{[Tsai 93]}),  and the main result of this article can be applied back to give a more intrinsic proof of Tsai's theorem. }

\normalsize


\section{Introduction}
\label{section 1}
Let $(Q^n,Q^m)$ be a pair of hyperquadrics of dimensions $n,m$ respectively ($n<m$) and suppose that $Q^m\subset\mathbb{P}^{m+1}$ is defined by $z_1^2+\cdots +z_m^2-2z_{m+1}z_{m+2}=0$. Then there is a natural embedding $i:Q^n\hookrightarrow Q^m$ defined by letting $z_{n+1}=\cdots =z_{m}=0$, i.e., $i$ makes $Q^n$ be expressed as $z^2_1+\cdots +z^2_n-2z_{m+1}z_{m+2}=0$ in $\mathbb{P}^{n+1}\subset\mathbb{P}^{m+1}$. $i(Q^n)\subset Q^m$ is a totally geodesic (or ``flat'' in non-technical words) submanifold. By the action of Aut$(Q^m)\cong SO(m+2,\mathbb{C})$, we obtain a (possibly non-totally geodesic) submanifold $g\circ i(Q^n)\subset Q^m$, we call it a \textsl{standard model} in $Q^m$ for any given $g\in$Aut$(Q^m)$ (cf. \cite{[MoZ]}) (of course when $g=$id$\in$Aut$(Q^m)$, the standard model is the flat one). In this article, we aim at proving the following theorem:

\textbf{Main Theorem} \textsl{Suppose $U\subset Q^m$ is an open subset and $S\subset U$ is a local $n$-dimensional complex submanifold of $Q^m$. If $S$ satisfies: 1. for any $x\in S,\ \ \mathbb{P}(T_xS)\cap\mathscr{C}_x(Q^m)\cong Q^{n-2}$; 2. germ of any minimal rational curve (MRC) $L\subset Q^m$ issuing from $x\in S$ such that $T_xL\subset T_xS$ also lies on $S$;  then $S$ is a subset of some standard model.}

$\mathscr{C}_x(Q^m)$ denotes the \textsl{Variety of Minimal Rational Tangents} (VMRT) of $Q^m$ at $x$ (for a comprehensive theory of VMRT, the reader may refer to \cite{[Mok 08]} or \cite{[HwM 99]}). It is well-known that $\mathscr{C}_x(Q^m)\cong Q^{m-2}$. The first condition in the Main Theorem is related to the notion of \textsl{sub-VMRT structure} formulated in \cite{[MoZ]}. In the case of pairs of hyperquadrics, this is a very loose condition. In fact, for a given hyperquadric $Q^{m-2}\subset\mathbb{P}^{m-1}$, a generic projective subspace $\mathbb{P}^{n-1}\subset\mathbb{P}^{m-1}$ intersects with $Q^{m-2}$ and produce a non-singular hyperquadric $Q^{n-2}\subset\mathbb{P}^{n-1}$. So a generic complex submanifold $S\subset Q^m$ will satisfy the first condition (or inherit a sub-VMRT structure modelled on $(Q^n,Q^m)$ in the sense of \cite{[MoZ]}) while $S$ is not necessarily any standard model. This simple fact concerning the ``flexibility'' of hyperquadric eflects that the pair $(Q^n,Q^m)$ is not rigid in the sense of \cite{[MoZ]}, we hope to establish a rigidity theorem $(Q^n,Q^m)$ in a sense weaker than \cite{[MoZ]}. This is why we introduce the second condition. In fact, if $S\subset Q^m$ is a subset of some standard model $Q^n\subset Q^m$, the property of the second condition naturally follows (cf. \cite{[Mok 89]})   

On the other hand, Tsai proved the isometric total geodesy (up to normalising constant) of the proper holomorphic mapping $f:\Omega_1\hookrightarrow\Omega_2$ between bounded symmetric domains $\Omega_1,\Omega_2$ whose rank satisfies rk$(\Omega_1)\geq$rk$(\Omega_2)\geq 2$ (cf. \cite{[Tsai 93]}). A key step in his proof is the following special case:


\begin{theorem}[proposition 2.1 \cite{[Tsai 93]}]
\label{Thm 1}
Suppose $f:D^{\Rmnum{4}}_{3}\rightarrow D^{\Rmnum{4}}_{n},\ \ n\geq 3$, is a proper holomorphic map. Then $f$ is totally geodesic isometric embedding (up to normalising constant).
\end{theorem}

$D^{\Rmnum{4}}_{n}$ denotes the type $\Rmnum{4}$ bounded symmetric domain of dimension $n$, it is the non-compact symmetric space dual to hyperquadric, i.e., the Borel embedding $i:D^{\Rmnum{4}}_{n}\hookrightarrow Q^n$ realises $D^{\Rmnum{4}}_{n}$ as an open subset of $Q^n$ contained in some Harish-Chandra coordinate chart $U\cong\mathbb{C}^n\subset Q^n$. Meanwhile, the minimal disc $\Delta\subset D^{\Rmnum{4}}_n$ is realised as an open subset of minimal rational curve $L\cong\mathbb{P}^1\subset Q^n$ in light of fact that $\Delta$ itself is the non-compact symmetric space dual to $\mathbb{P}^1$ (cf. \cite{[Mok 89]}). Proof of Theorem \ref{Thm 1} relies on the fact that minimal discs of the domain manifold $D^{\Rmnum{4}}_{3}$ is mapped biholomorphically to those of the target manifold $D^{\Rmnum{4}}_n$ due to properness (\cite{[MoT 92]}). Consider the composition $f'=i\circ f:D^{\Rmnum{4}}_3\hookrightarrow D^{\Rmnum{4}}_n\hookrightarrow Q^n$. Let $S$ be the image of $D^{\Rmnum{4}}_3$ under $f'$. Then $S\subset Q^n$ is a local complex submanifold satisfying both conditions in the Main Theorem provided that $f$ preserves minimal discs. In this way, it is hopeful that we can give a new and more intrinsic proof for Theorem \ref{Thm 1} through our Main Theorem without computing the second fundamental form.

\section{Geometry of a standard model in $Q^m$}
\label{section 2}
In this section, we first of all briefly recall some fundamental theories of $Q^n$ and $D^{\Rmnum{4}}_n$ as symmetric spaces. The reader may refer to \cite{[Mok 89]}, \cite{[Hel 78]}, \cite{[Wolf 72]}. Denote by Iso$(M)$ the group of isometry of symmetric manifold $M$. Then $G_0=$Iso$(D^{\Rmnum{4}}_n)$ and $G_c=$Iso $(Q^n)$ are non-compact and compact real forms of complex simple Lie group $G^{\mathbb{C}}=SO(n+2,\mathbb{C})$ respectively. $K=SO(n,\mathbb{R})\times SO(2,\mathbb{R})$ is the isotropy subgroup about some reference point $o$, making $D^{\Rmnum{4}}_n,\ \ Q^n$ the quotient spaces, i.e., $D^{\Rmnum{4}}_n=G_0/K,\ \ Q^n=G_c/K$. Let $\frak{g}_0,\frak{g}_c,\frak{g},\frak{k}$ be the corresponding Lie algebras of $G_0,G_c,G^{\mathbb{C}},K$, then:

\centerline{$\frak{g}_0=\frak{k}\oplus\frak{m},\ \ \frak{g}_c=\frak{k}\oplus\sqrt{-1}\frak{m},\ \ \frak{g}=\frak{k}^{\mathbb{C}}\oplus\frak{m}^{\mathbb{C}}.$}

The superscript $\mathbb{C}$ denotes complexification.  The isotropy sub algebra $\frak{k}=\frak{so}(m)+\frak{so}(2)$ has a one-dimensional centre $j$ which gives rise to a complex structure on $Q^m$ and $D^{\Rmnum{4}}_m$ and a decomposition $\frak{m}^{\mathbb{C}}=\frak{m}^-\oplus\frak{m}^+$, where $\frak{m}^-,\frak{m}^+$ are the eigenspaces of the the adjoint action ad$(j)$ with the eigenvalue $-i,+i$ respectively. Furthermore, $\frak{m}^-,\frak{m}^+$ are are Abelian subalgebras of $\frak{g}$ whose Abelian Lie subgroup is denoted by $M^-,M^+$ respectively. We are interested in the action of $M^-$ which produces non-totally geodesic (``non-flat'') standard model $Q^n$ of $Q^m$. 

Recall that we have expressed $Q^m=\{z^2_1+\cdots z^2_m-2z_{m+1}z_{m+2}=0\}\subset\mathbb{P}^{m+1}$ and the flat (totally geodesic) standard model $Q^n=\{z^2_1+\cdots+z^2_n-2z_{m+1}z_{m+2}=0\}\subset Q^m$. From now on let $o$ denote the reference point $o=[0,...,0,1,0]\in Q^m$ unless otherwise stated. With respect to $o$, the subgroup $M^-$ is expressible in the following form:

\begin{center}
$ $$\left[\begin{array}{cc}
 I_m &     B\\
 C      & D
\end{array}\right]
$$ $, \ \
\end{center}

\medskip
 
 where 
\medskip
\begin{tabular}{lll}
$ B= \sqrt{2} $$\left[\begin{array}{cc}
 0 & a_1 \\
 \vdots & \vdots \\
 0 & a_m 
\end{array}\right]
$$,$ 
& 
$C= \sqrt{2}$$\left[\begin{array}{ccc}
 a_1 & \cdots & a_m \\
 0 & \cdots & 0 
\end{array}\right]
$$, $
&
$D=$$\left[\begin{array}{cc}
 1 & a^2_1+\cdots +a^2_m \\
 0 & 1 
\end{array}\right]
$$ $.

\end{tabular} for $a_1,...,a_m\in\mathbb{C}$. It can be seen that there exists a unique parameter set $\{a_1,...,a_m\}$ associated to any  $g\in M^-$. 

Fix a Harish-Chandra coordinate chart $W\subset Q^m$ containing $o$, where $W=\{[z_1,...,z_{m+2}]\in Q^m| z_{m+1}\neq 0\}\cong\mathbb{C}^m$.  By abuse of notation, we denote by $(z_1,...,z_m)$ the Harish-Chandra coordinate in the open set $W$. In this sense, $o=(0,...,0)\in W$. Then the flat standard model $Q^n$ is expressed as $\{(z_1,...,z_m)|z_{n+1}=\cdots =z_{m}=0\}\cong\mathbb{C}^n$ if it is restricted to $W$. For some $g\in M^-$with parameter $\{a_1,...,a_m\}$, the flat standard model $Q^n$ is transformed in the following way in terms of the Harish-Chandra coordinate in $W$:

\begin{align*}
&z'_i=z_i+\sqrt{2}a_iz_{m+2},\ \ 1\leq i\leq n\\
&z'_l=\sqrt{2}a_lz_{m+2},\ \ n+1\leq l\leq m\\
&z'_{m+1}=\sqrt{2}\sum\limits_{i=1}^{n}{a_iz_i}+1+(a^2_1+\cdots+a^2_m)z_{m+2}\\
&z'_{m+2}=z_{m+2}
\end{align*} where $z_{m+2}=\frac{1}{2}\sum\limits_{i=1}^{n}{z^2_i}$. From now on, let $i,j,k$ always range in $\{1,...,n\}$ and $l$ always range in $\{n+1,...,m\}$ until otherwise stated. The above computation immediately shows:


\begin{lemma}
\label{Lemma 2.1}
Denote by $\mathscr{S}(Q^n)$ the subgroup of $M^-$ which leaves the flat standard model $Q^n$ invariant as a set, then $\mathscr{S}(Q^n)$ is of the form $ $$\left[\begin{array}{cc}
 I_m &     B'\\
 C'      & D'
\end{array}\right]
$$ $\ \ where $B',C',D'$ is obtained from $B,C,D\ \ \textnormal{by making} a_{n+1}=\cdots =a_m=0$. Moreover, dim$(M^-/\mathscr{S}(Q^n))=m-n$, $M^-/\mathscr{S}(Q^n)$ is parameterized by $\{a_{n+1},...,a_m\}$.
\end{lemma}

\textbf{Remark} $M^-$ does not leave the flat standard model $Q^n$ invariant as a set, such flat $Q^n$ is not \textsl{invariant geodesic} submanifold of $Q^m$ in the sense of definition 4.1 of \cite{[Tsai 93]}.

By virtue of lemma \ref{Lemma 2.1}, we denote by $M(a_{n+1},...,a_m)$\ \ the non-flat standard model obtained from the action of $g\in M^-\subset$Aut$(Q^m)$\ \ on the flat $Q^n$ where $g$ is parametrised by $\{a_{n+1},...,a_m\}$ and the non-flat $Q^n$ is uniquely determined by the parameter set $\{a_{n+1},...,a_m\}$. We can assume without loss of generality that all of $\{a_{n+1},...,a_m\}$ are non-zero, for if some $a_l=0,\ \ l\in\{n+1,...,m\}$, then $z'_l$ vanishes, which means it is reduced to considering the pair $(Q^n,Q^{m-1})$. Dividing all $z'_is,z'_ls$\ \ by $z'_{m+1}$\ \ whenever a small neighborhood $U\subset Q^n$\ \ of $o$\ \ is chosen such that $z'_{m+1}\neq 0$,\ \ routine computation shows:


\begin{lemma}
\label{Lemma 2.2}
Assuming all of $a_{n+1},...,a_m$\ \ are non-zero, a small neighborhood $ V\subset M(a_{n+1},...,a_m)$\ \ around $o$\ \ is representable in $W\cong\mathbb{C}^m$\ \ as a subvariety defined by polynomials $\{(z_1,...,z_m)\in \mathbb{C}^m|z^2_1+\cdots+z^2_m=\frac{\sqrt{2}}{a_l}z_l,\ \ l=n+1,...,m\}$.\ \ Solving $z_{n+1},...,z_m$\ \ in terms of $z_1,...,z_n$,\ \ it can be expressed as a graph $\{(z_1,...,z_n,g_{n+1},...,g_m)\}$\ \ in $W$\ \ over a small neighborhood $U\subset Q^n$\ \ near $o$\ \ where

\begin{center}
 $2g_l(a^2_{n+1}+\cdots+a^2_m)=\sqrt{2}a_l-\sqrt{2a^2_l-4a^2_l(a^2_{n+1}+\cdots+a^2_m)(z^2_1+\cdots+z^2_n)}$.
\end{center} Performing Taylor expansion near 0 about $\omega=z^2_1+\cdots+z^2_n$

\begin{center}
$g_l=\frac{1}{2}\omega(\sqrt{2}a_l+\frac{1}{\sqrt{2}}aa_l\omega+\cdots)\ \ \ \ \ \ \ \ \ \ \ \ \ (*)$
\end{center} where $a=a^2_{n+1}+\cdots+a^2_m$.
\end{lemma}

For clarity, we fix the meanings of the notations $W,U,V$ for open sunsets on the ambient manifold $Q^m$ defined as above. Equipped with these preparations, we go back to our original settings. Suppose a fixed germ of complex submanifold $S$\ \ lying in $W$\ \ satisfying the conditions in Main Theorem. With the help of translation $M^+$=exp$(\frak{m}^+)$\ \ and linear transformation $K^{\mathbb{C}}$=exp$(\frak{k}^{\mathbb{C}})$,\ \ we can further assume without loss of generality that $o\in S$\ \ and $T_oS=T_oQ^n$.


\begin{proposition}
\label{proposition 2.3}
Suppose a complex submanifold $S\subset Q^m$\ \ satisfies the conditions in Main Theorem, then for any fixed $p\in S$,\ \ there exists a unique standard model $M(a_{n+1},...,a_m)$\ \ for some $\{a_{n+1},...,a_m\}$\ \ passing through $p$,\ \ which is tangent to $S$\ \ to order 2 at $p$.
\end{proposition}

\begin{proof}

Suppose $p=o,\ \ T_oS=T_oQ^n$.\ \ Evidently $S$\ \ is expressible as a graph over $U$\ \ as $(z_1,...,z_n,f_{n+1},...,f_m)$,\ \ where $f_ls$\ \ are holomorphic functions in $(z_1,...,z_n)$\ \ defined on $U$\ \ with all the first order derivatives vanish at $o$,\ \ i.e., $\frac{\partial f_l}{\partial z_i}(0)=0$.\ \ Then $\mathscr{C}_o(S)=\mathbb{P}(T_oS)\cap \mathscr{C}_o(Q^m)=\{[\lambda_1,...,\lambda_m]\in\mathbb{P}(T_oQ^m)|\lambda^2_1+\cdots+\lambda^2_n=0,\ \ \lambda_{n+1}=\cdots=\lambda_m=0\}=\mathscr{C}_o(Q^n)\cong Q^{n-2}$ 

For minimal rational curves $L$\ \ passing through $o$\ \ with $T_oL=(\lambda_1,...,\lambda_n)\in \widetilde{\mathscr{C}}_o(S)$,\ \ since $L\subset S$, we have for small $z_is$\ \ with $z^2_1+\cdots+z^2_n=0,\ \ f_l(z_1,...,z_n)\equiv 0$.\ \ Thus we get factorization of $f_ls$\ \ on $U$\ \ (may be shrunk if necessary):

\begin{center}
$f_l=\frac{1}{2}(z^2_1+\cdots+z^2_n)h_l(z_1,...,z_n)\ \ \ \ \ \ \ \ \ \ \ \ \ \ \ (**)$
\end{center} 

where $h_ls$\ \ are holomorphic on $U$. 

\begin{center}
$\frac{\partial^2 f_l}{\partial z_j \partial z_k}(0)=h_l(0)\delta_{jk},\ \ \delta_{jk}$\ \ being Kronecker symbol.
\end{center}

By virtue of the expression (*), 

\begin{center}
$\frac{\partial^2 g_l}{\partial z_j \partial z_k}(0)=\sqrt{2}a_l\delta_{jk}$
\end{center}

Suppose some standard model $M(a_{n+1},...,a_m)$\ \ passing through $o$\ \ where $a_l=\frac{h_l(0)}{\sqrt{2}}$,\ \ then $M(\frac{h_{n+1}(0)}{\sqrt{2}},...,\frac{h_m(0)}{\sqrt{2}})$\ \ is tangent to $S$\ \ to order 2 at $o$\ \ in the sense of $\frac{\partial^2 g_l}{\partial z_j \partial z_k}(0)=\frac{\partial^2 f_l}{\partial z_j \partial z_k}(0)=\sqrt{2}a_l\delta_{jk}$\ \ when the standard model is expressed as a graph over $U$\ \ according to Lemma \ref{Lemma 2.2}. 

\end{proof}

\section{Proof of the Main Theorem}
\label{section 3}

Now we come to the stage of proving our Main Theorem. If we keep the settings and notations laid out in the previous section, evidently it suffices to establish the following theorem:

\begin{theorem}
\label{theorem 3.1}
$S$\ \ is an open subset of the standard model $M(\frac{h_{n+1}(0)}{\sqrt{2}},...,\frac{h_m(0)}{\sqrt{2}})$. This standard model is non-flat unless $h_{n+1}(0)=\cdots =h_{m}(0)=0$.
\end{theorem}

Denote by $M$ the unique standard model $M(\frac{h_{n+1}(0)}{\sqrt{2}},...,\frac{h_m(0)}{\sqrt{2}})$ with second order tangency to $S$ at $o$. We establish Theorem \ref{theorem 3.1} based on the idea of \textsl{adjunction of minimal rational curves} (cf. \cite{[HoM 10]}). This adjunction process relies on the notion of \textsl{parallel transport of VMRT} along minimal rational curves (cf. \cite{[HoM 10]}). The proof of Theorem \ref{theorem 3.1} goes essentially in the direction of establishing the parallel transport of VMRT for the case of hyperquadrics.

\begin{proof}
In $W\cong \mathbb{C}^m$,\ \ we always identify $T_xQ^n,\ \ T_yQ^m$\ \ with some fixed $\mathbb{C}^n,\ \ \mathbb{C}^m$\ \ at $\forall x\in U,\ \ \forall y\in W$,\ \ respectively. Fix some $\alpha=(\alpha_1,...,\alpha_n,0...,0)\in \widetilde{\mathscr{C}}_o(S)$,\ \ i.e. $\sum\limits_i{\alpha^2_i=0}$\ \ and a line $L(t)$\ \ passing through $o=L(0)$\ \ with $T_oL=\alpha$\ \ parameterized by small $t,\ \ L(t):=(t\alpha_1,...,t\alpha_n,0...,0)\subset U$.\ \ We claim that for any fixed $\alpha\in\widetilde{\mathscr{C}}_o(S),\ \ h_l(t\alpha)$\ \ is constant for small $t$,\ \ i.e., $h_l(z)\equiv h_l(0)=\sqrt{2}a_l$\ \ along lines. 

Fix any $t_0\neq 0$\ \ sufficiently small and consider $\widetilde{\mathscr{C}}_{L(t_0)}(S)=\widetilde{\mathscr{C}}_{L(t_0)}(X)\cap T_{L(t_0)}Q^m$\ \ parameterized by $(\lambda_1,...,\lambda_n)\in T_{L(t_0)}Q^n$\ \ as follows: 

\begin{center}
$\{(\lambda_1,...,\lambda_n)\in \mathbb{C}^n|\sum\limits_{i}{\lambda^2_i}+\sum\limits_{l}(\sum\limits_i\lambda_i\frac{\partial f_l}{\partial z_i}(t_0\alpha))^2=0\}\ \ \ \ \ \ \ \ \ \ \ \ (\dag)$
\end{center}

By virtue of factorisation (**) (see the proof of Proposition \ref{proposition 2.3}), $\frac{\partial f_l}{\partial z_i}=z_ih_l+\frac{1}{2}(z^2_1+\cdots+z^2_n)\frac{\partial h_l}{\partial z_i}$,\ \ and the fact that $\sum\limits_i \alpha^2_i=0$,\ \ we can reduce $\dag$\ \ to:

\begin{center}
$\{(\lambda_1,...,\lambda_n)\in \mathbb{C}^n|\sum\limits_i \lambda^2_i+t_0^2(\sum\limits_l h^2_l(t_0\alpha))(\sum\limits_i \alpha_i\lambda_i)^2=0\}$\ \ \ \ \ \ \ \ \ \ $(\dag')$
\end{center}

For $\forall \lambda=(\lambda_1,...,\lambda_n)$\ \ satisfying $\dag',$

\begin{center}
$(t_0\alpha_i+s\lambda_i;f_l(t_0\alpha+s\lambda))\subset W$\ \ for small $s$
\end{center}

is a germ of line parameterized by $s$\ \ lying on $S$\ \ for $S$\ \ is line preserving. This implies:

\begin{center}
$\frac{d^2}{ds^2}f_l(t_0\alpha+s\lambda)|_{s=0}=\sum\limits_{i,j}\frac{\partial ^2 f_l}{\partial z_i \partial z_j}(t_0\alpha)\lambda_i\lambda_j=0\ \ \ \ \ \ \ \ (\dag\dag)$
\end{center}

Again by factorization (**), 

\begin{center}
$\frac{\partial^2 f_l}{\partial z_i\partial z_j}=\delta_{ij}h_l+z_i\frac{\partial h_l}{\partial z_j}+z_j\frac{\partial h_l}{\partial z_i}+\frac{1}{2}(z^2_1+\cdots+z^2_n)\frac{\partial^2 h_l}{\partial z_i\partial z_j}$
\end{center}

and $\alpha^2_1+\cdots+\alpha^2_n=0$,\ \ we can reduce $\dag\dag$\ \ to 

\begin{center}
$(\sum\limits_i \lambda^2_i)h_l(t_0)+2t_0(\sum\limits_i \alpha_i\lambda_i)(\sum\limits_j \lambda_j\frac{\partial h_l}{\partial z_j}(t_0))=0\ \ \ \ \ \ (\dag\dag')$
\end{center}

$\alpha=(\alpha_1,...,\alpha_n,0,...,0)\in T_oS$\ \ is fixed. For $\forall t$\ \ small, we can find $n$\ \ linearly independent vectors $\lambda^1,...,\lambda^n\in\mathbb{C}^n$\ \ each $\lambda^k=(\lambda^k_1,...,\lambda^k_n)$\ \ satisfies $\dag'$\ \ and $\sum\limits_i \lambda^k_i\alpha_i\neq 0,\ \ k=1,2,...,n$.

$\dag\dag'$\ \ combined with $\dag'$\ \ then gives

\begin{center}
$\sum\limits_i \lambda^k_i\frac{\partial h_l}{\partial z_i}(t\alpha)=th_l(t\alpha)(\sum\limits_{p=n+1}^m h^2_p(t\alpha))(\sum\limits_i \alpha_i\lambda^k_i),\ \ k=1,2,...,n$.
\end{center}

This implies nothing but 
\begin{center}
$\frac{\partial h_l}{\partial z_i}=th_l(t\alpha)(\sum\limits_{p=n+1}^m h^2_p(t\alpha))\alpha_i\ \ \ \ \ \ \ \ \ (\dag\dag\dag)$.
\end{center}

On the other hand, 

\begin{center}
$\frac{dh_l}{dt}(t\alpha)=\sum\limits_i{\frac{\partial h_l}{\partial z_i}\alpha_i}=th_l(t\alpha)(\sum\limits_{i}{\alpha_i^2})(\sum\limits_{p=n+1}^m{h^2_p(t\alpha)})=0$.
\end{center}

So 

\begin{center}
$h_l(t\alpha)\equiv h_l(0)=\sqrt{2}a_l\ \ \ \ \ \ \ \ \ \ \ (\dag\dag\dag')$,
\end{center}
 
Proving the claim. By $\dag$,\ \ we show that 

\begin{center}
$\mathscr{C}_{L(t)}(S)=\mathscr{C}_{L(t)}(M),\ \ L(t)=(t\alpha_1,..,t\alpha_n,0,...,0)\in U$,
\end{center}

both parameterized by $(\lambda_1,...,\lambda_n)\in T_{L(t)}Q^n$\ \ as 

\begin{center}
$\{(\lambda_1,...,\lambda_n)\in \mathbb{C}^n|\sum\limits_i \lambda^2_i+2t^2(\sum\limits_l a^2_l)(\sum\limits_i \alpha_i\lambda_i)^2=0\}$
\end{center}

Thus we have established the identification of VMRTs of $S$\ \ and $M$\ \ at any other point $x\in L,\ \ x\neq o$\ \ on any line $L$\ \ issuing from $o$.\ \ So by line preservation property of $S$ (i.e., the second condition in Main Theorem), any line $L'$\ \ issuing from $\forall x\in L$\ \ is contained both in $S$\ \ and $M$.\ \ To prove this theorem, it suffices to show that $M$\ \ is tangent to $S$\ \ at $L(t)$\ \ to order 2 for $\forall t\neq 0$.\ \ If this is done, we can repeat the above argument which finally leads us to the conclusion that $S$ is identified with some open subset of $M$, as a result of the process of adjunction of minimal rational curves. For this, it suffices to prove:

\begin{center}
$\frac{\partial^2 g_l}{\partial z_i \partial z_j}(t\alpha)=\frac{\partial^2 f_l}{\partial z_i \partial z_j}(t\alpha),\ \ t\neq 0$.
\end{center}

This just follows from direct computation and combination of $\dag\dag\dag$\ \ and $\dag\dag\dag'$:

\begin{align*}
\frac{\partial^2 f_l}{\partial z_i \partial z_j}(t\alpha)&=\delta_{ij}h(t\alpha)+t\alpha_i\frac{\partial h_l}{\partial z_j}+t\alpha_j\frac{\partial h_l}{\partial z_i}\\
&=\sqrt{2}a_l\delta_{ij}+2\sqrt{2}\alpha_i\alpha_ja_l(a^2_{n+1}+\cdots+a^2_m)\\
&=\frac{\partial^2 g_l}{\partial z_i \partial z_j}(t\alpha)
\end{align*}

\end{proof}

Our Main Theorem allows us to give a new proof of Theorem \ref{Thm 1}, to which Tsai reduce his main result in \cite{[Tsai 93]}. Tsai's original proof relies on computing second fundamental form with respect to canonical metrics on bounded symmetric domains. Our new proof has the merit of being free from computing second fundamental forms, hence more conceptual.

\begin{proof}
By considering the radial limits of the proper mapping $f$ to the boundary of the target $D^{\Rmnum{4}}_n$ and using Cauchy integral formula along the boundary, Mok and Tsai proved in \cite{[MoT 92]} that $f$ maps the minimal discs of $D^{\Rmnum{4}}_3$ to minimal discs of $D^{\Rmnum{4}}_n$, i.e., $f$ maps germs of minimal rational curves of domain manifold $Q^3$ to those of the target manifold $Q^n$ by regarding bounded symmetric domains as open subsets of their compact duals. In view of this, we can actually regard $f$ as restriction of some $\alpha\in$Aut$(Q^n)$ to $D^{\Rmnum{4}}_3$ such that the boundary is preserved. Besides, $f$ is equivariant. We claim that $\alpha$ is actually an isometry in Iso$(D^{\Rmnum{4}}_n)$. For this, consider the (unique) canonical metrics $g,h$ for $D^{\Rmnum{4}}_3,D^{\Rmnum{4}}_n$ respectively. Through the pull-back, $g-f^*h$ (after normalization) defines an covariant $(1,1)$ tensor on $D^{\Rmnum{4}}_3$ which vanishes along minimal rational direction $\alpha$, i.e., $g(\alpha,\overline{\alpha})-f^*h(\alpha,\overline{\alpha})=0$ because $f$ is biholomorphism on minimal discs, hence isometry along minimal direction due to properness. Polarization arguments again yields the vanishing of this $(1,1)$ tensor, hence $g=f^*h$, proving our claim of isometry. It is well known that Iso$(D^{\Rmnum{4}}_n)\cap M^-=$id, i.e., $\alpha\notin M^-$ unless it is identity. While isometry group is linear action, then $f(D^{\Rmnum{4}}_3)$ is an affine linear submanifold in $D^{\Rmnum{4}}_n$, hence totally geodesic when we take the equivariance of $f$ into account. 

\end{proof}

\end{document}